\documentclass[oneside]{amsart}
\usepackage[left=2cm,top=2cm,right=3cm,bottom=1cm, footskip=.5cm]{geometry}
\usepackage{amsmath}
\usepackage{graphicx}
\usepackage{tikz}
\usepackage{fancyvrb}
\usepackage{pdfpages}
\usepackage[makeroom]{cancel}
\usepackage{amssymb} 
\usepackage[parfill]{parskip}
\usepackage{mathtools}
\usepackage{hyperref}
\usepackage{tikz-cd}
\usepackage{setspace}
\usepackage{amsthm}
\usepackage{cite}

\newtheoremstyle{colon}%
{}
{}
{\itshape}
{}
{\bfseries}
{:}
{ }
{}

\theoremstyle{colon}

\newtheorem{lemma}{Lemma}[]
\newtheorem{theorem}[lemma]{Theorem}

\newtheorem{definition}[lemma]{Definition}

\newtheorem{remark}[lemma]{Remark}

\title{The equivariant covering homotopy property}
\author{Andrew Ronan}
\date{}
\begin{document}
	
	\maketitle
	
	\begin{abstract}
		In this paper, we explain how the more general context of generalised equivariant bundles allows for a simple proof of the ECHP, which makes use of induction on the dimension / no. of connected components of compact Lie groups. We also make clear the link between the ECHP and the theory of Hurewicz fibrations.
	\end{abstract}

\section{Introduction} 

Generalised equivariant bundles were first introduced by Lashof and May in \cite{LM86} to encompass examples of equivariant maps which morally should be bundles, but did not fit into the previously existing frameworks of \cite{L82}, \cite{D69} or \cite{N78}. The canonical example is the quotient of a compact Lie group by a closed normal subgroup, $q: \Gamma \to \Gamma / \Pi$, which should be a principal bundle, but other examples are given in \cite[Examples 1 and 2]{LM86}. It was noted on \cite[pg.175]{LMS86} that working with generalised equivariant bundles is often notationally simpler than working with the less general definitions, and our work here is certainly consistent with that observation. However, as will exhibit with a new inductive proof of the ECHP, \cite[Corollary 8]{LM86}, working with generalised equivariant bundles has a number of advantages which extend beyond notational simplicity.

The present paper has two objectives. The primary objective, which is the main original contribution, is the aforementioned new proof of the ECHP. In particular, we show that the more general context of generalised equivariant bundles allows for a simple proof of the ECHP, which makes use of induction on the dimension / no. of connected components of compact Lie groups at a key step and illuminates the underlying mathematics. Our proof does not depend on Bierstone's condition (\cite[pg.258]{L82}, \cite[pg. 620]{B73}), as is the usual route, which is a property specific to the less general $G$-$A$ (or $(A; A \times G)$) bundles of \cite{L82} (see also Remark \ref{echp66}). Another bonus of the proof is that we do not require any additional local patching result specific to bundles, such as in \cite[Ch. 4.9]{H93}, instead making use of the local patching property of Hurewicz fibrations (Lemma \ref{echp9} below) which will already be familiar to most homotopy theorists. Note, however, that a local to global gluing result of some variety is still necessary for the proof of the ECHP, and the proof of Lemma \ref{echp9} is just as technical as the proof in \cite[Ch. 4.9]{H93}.  We also note that the conclusion of the classical proof of the ECHP is particularly complex and confusing, involving the identification of $G \to G/H$ as an $(H \times H; H \times H \times H)$-bundle with fibre $H$ and the paracompactness of the path space $P(G/H)$, see \cite[pg. 266]{L82}.

The second objective is to give a self-contained introduction to the theory of generalised equivariant bundles, which we hope will convince the reader of the simplicity of the ideas. Therefore, we give clear definitions of all the concepts we use in our proof of the ECHP, with the aim of making this introduction accessible to as wide an audience as possible. The interested reader can then learn about the relationship between universal principal bundles and equivariant Eilenberg-MacLane spaces from \cite[V.4, VII.2]{M96}, and the fixed point structure of principal bundles from \cite[Theorem 12]{LM86}. 

\subsection{Notations and Conventions}
We work throughout with a continuous extension of compact Lie groups:	
\[
1 \to \Pi \to \Gamma \to G \to 1
\]	
and we assume that all subgroups of compact Lie groups under discussion are closed. We also work in the category of CGWH spaces, \cite{S09}.

\subsection{Acknowledgments} I would like to thank the referee and editor for a number of helpful suggestions.

\section{Generalised equivariant bundles and the ECHP}

In this section, we introduce generalised equivariant bundles and the ECHP. We begin with our favoured definition of a $(\Pi; \Gamma)$-bundle with fibre $F$, which makes clear that the bundles under discussion really are a simple generalisation of the `non-equivariant' $(\Pi; \Pi)$ case, compare \cite[Definition 2.4]{S51}:

\begin{definition} \label{echp24}
	Let $F$ be a $\Gamma$-space on which $\Pi$ acts effectively and let $\sigma: D \to B$ be a map of $G$-spaces. A chart for $\sigma$ is a pair $(U, \psi)$ such that $U$ is an open subset of $B$ and $\psi: U \times F \to \sigma^{-1}(U)$ is a homeomorphism over $U$, also known as a trivialisation: 
	
	\[
	\begin{tikzcd}
		U \times F \arrow{rr}{\psi} \arrow{dr} & & \sigma^{-1}(U) \arrow{dl} \\
		& U & 
	\end{tikzcd}	
	\]
	
	By an atlas for $\sigma$, we mean a set of charts for $\sigma$, $\mathcal{C}$, such that:
	
	\begin{enumerate}
		\item for every $b \in B$, there exists a chart $(U, \psi) \in \mathcal{C}$ with $b \in U$,
		\item if $(U_i,\psi_i), (U_j,\psi_j) \in \mathcal{C}$, $\gamma \in \Gamma$, $g$ is the image of $\gamma$ in $G$, $u \in U_i$ and $gu \in U_j$, then the composite:	
		\[
		F \cong \{u\} \times F \xrightarrow{\gamma^{-1}} \{u\} \times F \xrightarrow{\psi_j^{-1} g \psi_i} \{gu\} \times F \cong F
		\]
		
		is multiplication by some $\pi \in \Pi$.
	\end{enumerate}
	
A maximal atlas for $\sigma$ is an atlas for $\sigma$ which is not a subset of any other atlas for $\sigma$.

A $(\Pi; \Gamma)$-bundle with fibre $F$, $(\sigma, \mathcal{C})$ is a map of $G$-spaces $\sigma: D \to B$ equipped with a maximal atlas for $\sigma$. Note that the maximal atlas, $\mathcal{C}$, is often dropped from the notation.
	
	Given a $(\Pi; \Gamma)$-bundle with fibre $F$, $(\sigma, \mathcal{C})$, a map $F \to D$ is said to be admissible if it factors through a fibre $\sigma^{-1}(b)$ and the composite with any trivialisation $F \to \sigma^{-1}(b) \xrightarrow{\psi_i^{-1}} \{b\} \times F \cong F$ is multiplication by some $\pi \in \Pi$.
	
	Finally, given $(\Pi; \Gamma)$-bundles with fibre $F$, $(\sigma_1: D_1 \to B_1, \mathcal{C}_1)$ and $(\sigma_2: D_2 \to B_2, \mathcal{C}_2)$, a map $(\sigma_1, \mathcal{C}_1) \to (\sigma_2, \mathcal{C}_2)$ is defined to be a $G$-map, $D_1 \to D_2$, that takes admissible maps to admissible maps. 
	
\end{definition}

\begin{remark} \label{echp60} 
	If $(\sigma, \mathcal{C})$ is a $(\Pi; \Gamma)$-bundle with fibre $F$ and $(U_i, \psi_i), (U_j, \psi_j) \in \mathcal{C}$, then $\psi_j^{-1} g \psi_i \gamma^{-1}$ defines a transition map $U_i \cap g^{-1}(U_j) \to Aut(F)$. Since $F$ is a $\Gamma$-space on which $\Pi$ acts effectively, we have an injective map $\Pi \to Aut(F)$, which is a homeomorphism onto its image since $\Pi$ is compact Hausdorff and we are working in the category of CGWH spaces, \cite[Lemma 1.4(b)]{S09}. Condition (2) of Definition \ref{echp24} then tells us that the transition map $U_i \cap g^{-1}(U_j) \to Aut(F)$ has image in $\Pi$ and so defines a continuous map $U_i \cap g^{-1}(U_j) \to \Pi$. This explains why we do not need a continuity condition akin to \cite[pg.8 Condition (10)]{S51} in our definition of a $(\Pi; \Gamma)$-bundle with fibre $F$.
\end{remark}

If $(\sigma: D \to B, \mathcal{C})$ is a $(\Pi; \Gamma)$-bundle with fibre $F$, we can define an atlas for $\sigma \times 1: D \times I \to B \times I$ by $\mathcal{C}^{'} := \{(U \times I, \psi \times 1) | (U,\psi) \in \mathcal{C}\}$. We then view $\sigma \times 1$ as a $(\Pi; \Gamma)$-bundle with fibre $F$ equipped with the unique maximal atlas, which we call $\mathcal{C} \times 1$, containing $\mathcal{C}^{'}$.
As in \cite[Corollary 3.3]{B73}, the ECHP for $(\Pi; \Gamma)$-bundles with fibre $F$ is defined as follows:

\begin{definition} \label{echp20}
Let $\sigma_2 : D_2 \to B_2$ be a $(\Pi;\Gamma)$-bundle with fibre $F$. We say that $\sigma_2$ satisfies the ECHP if whenever:
	
	\[
	\begin{tikzcd}
		D_1 \arrow{r}{f} \arrow[swap]{d}{\sigma_1} & D_2 \arrow{d}{\sigma_2} \\
		B_1   \arrow[swap]{r}{h} & B_2 
	\end{tikzcd}
	\]
	
	is a map of $(\Pi; \Gamma)$-bundles with fibre $F$, and $H: B_1 \times I \to B_2$ is a $G$-homotopy starting at $h$, then there exists a map of $(\Pi; \Gamma)$-bundles with fibre $F$ of the form:
	
	\[
	\begin{tikzcd}
		D_1 \times I \arrow{r}{K} \arrow[swap]{d}{\sigma_1 \times 1} & D_2 \arrow{d}{\sigma_2} \\
		B_1 \times I  \arrow[swap]{r}{H} & B_2 
	\end{tikzcd}
	\]
	
	which restricts to the original map of bundles at $0 \in I$.
\end{definition}

The main theorem concerning the ECHP is then as follows:

\begin{theorem} \label{echp22}
	A numerable $(\Pi; \Gamma)$-bundle with fibre $F$, see Definition \ref{echp65}, satisfies the ECHP.
\end{theorem}
The term `equivariant covering homotopy property' is also used in the context of Hurewicz $G$-fibrations, which we recall are defined as follows:

\begin{definition} \label{echp64} 
	A map $p: E \to B$ of $G$-spaces is said to be a Hurewicz $G$-fibration (shorthand h$G$-fibration) if it has the right lifting property, in the category of $G$-spaces, with respect to all maps of the form $X \times \{0\} \xrightarrow{1 \times i_0} X \times I$, where $X$ is a $G$-space.
\end{definition}

It is sometimes said that $p$ is a Hurewicz $G$-fibration if it `satisfies the ECHP', for example on \cite[pg.545]{BH02}. However, in this paper, we will only use the term ECHP when it refers to covering homotopy properties of bundles. Note also that at least some caution is required here. Let $\sigma_2: D_2 \to B_2$ be a $(\Pi; \Gamma)$-bundle with fibre $F$. Then, as we will see, if $\sigma_2$ satisfies the ECHP of Definition \ref{echp20}, then $\sigma_2$ is a Hurewicz $G$-fibration. Indeed, $\sigma_2$ satisfies the ECHP of Definition \ref{echp20} iff its associated principal bundle is an $h\Gamma$-fibration, and we can then apply Lemma \ref{echp5}. However, as far as we are aware, it is not true that if $\sigma_2$ is a Hurewicz $G$-fibration, then $\sigma_2$ necessarily satisfies the ECHP of Definition \ref{echp20} - therefore, the ECHP and the definition of a Hurewicz $G$-fibration do not coincide in this context. Note that a counterexample would have to be somewhat pathological since if $B_2$ is paracompact Hausdorff, and the total space of the associated principal bundle is completely regular, then $\sigma_2$ is numerable, by \cite[Propositions 4 and 5]{LM86}, and therefore satisfies the ECHP. Note that the definition of paracompactness in use in \cite{LM86} assumes that the space is also Hausdorff and so subordinate partitions of unity exist.

On the other hand, the situation for principal $(\Pi; \Gamma)$-bundles is simpler, where:

\begin{definition} \label{echp21}
	A principal $(\Pi; \Gamma)$-bundle is a $\Pi$-free $\Gamma$-space, $E$, such that there is a $\Pi$-equivariant open cover $\{U_i\}$ of $E$ such that, for every $i$, $U_i \cong \Pi \times V_i$ as $\Pi$-spaces, for some $V_i$ with trivial $\Pi$-action. Define $B := E / \Pi$ and let $p: E \to B$ denote the quotient map. A map between principal $(\Pi;\Gamma)$-bundles is defined to be a $\Gamma$-map between the underlying $\Gamma$-spaces.
\end{definition}

\begin{remark} \label{echp63}
	 Our convention is that $\Pi$ acts on itself by right multiplication, $\pi \cdot x = x \pi^{-1}$. This choice is made so that $\Pi \times_{\Pi} F$ agrees with the usual definition. Note, however, that if $\Pi$ acts on itself by left multiplication instead, this would not change the definition of a principal $(\Pi; \Gamma)$-bundle.
\end{remark}

\begin{remark} \label{echp36}
	We prefer to define principal $(\Pi; \Gamma)$-bundles as spaces rather than as maps in order to emphasise the useful perspective that the category of principal $(\Pi; \Gamma)$-bundles can be identified with a full subcategory of the category of $\Gamma$-spaces. This is a straightforward category to work with, and we feel it provides the simplest definition of principal bundles. For example, it is clear from this definition why universal principal bundles can be identified with equivariant Eilenberg-MacLane spaces, \cite[V.4, VII.2]{M96}. Nevertheless, at times we identify a principal $(\Pi; \Gamma)$-bundle, $E$, with the quotient map $p: E \to B := E / \Pi$. Of course, all definitions define equivalent categories of principal bundles.
\end{remark}

 We have a 1-1 correspondence between $(\Pi; \Gamma)$-bundles with fibre $F$ and principal $(\Pi; \Gamma)$-bundles, see \cite[pg. 179]{LMS86} or \cite[Theorem 2.29]{Z21}, which we now prove for the sake of completeness:

\begin{theorem} \label{echp61}
	There is an equivalence of categories between the category of principal $(\Pi; \Gamma)$-bundles and the category of $(\Pi; \Gamma)$-bundles with fibre $F$.
\end{theorem}
 
 \begin{proof}
 	If $F$ is empty, then we must have $\Pi = 1$ and $\Gamma = G$, in which case it is clear from the definitions that both categories are equivalent to the category of $G$-spaces. In particular, 	if $(\sigma: D \to B, \mathcal{C})$ is a $(\Pi; \Gamma)$-bundle with fibre $\emptyset$, we define the associated principal $(\Pi; \Gamma)$-bundle, $Ass(\sigma, \mathcal{C})$, to be $B$. 
 	
 Now assume that $F$ is non-empty and $x \in F$.	If $(\sigma: D \to B, \mathcal{C})$ is a $(\Pi; \Gamma)$-bundle with fibre $F$, then define the associated principal $(\Pi; \Gamma)$-bundle, $Ass(\sigma, \mathcal{C})$, to be the subspace of admissible maps in $Map(F,D)$. Then $\Gamma$ acts by conjugation on $Ass(\sigma, \mathcal{C})$, by condition ii) of Definition \ref{echp24}, and the action is $\Pi$-free since $\Pi$ acts effectively on $F$. If $(U, \psi) \in \mathcal{C}$, note that evaluation at $x$ helps define a map $Ass(\sigma, \mathcal{C}) \to B$ and the preimage of $U$ is $\Pi$-homeomorphic to $U \times \Pi \subset Map(F,U \times F)$.  For the point set topology here, note that, for any space $Z$, the functors $Z \times -$ and $Map(Z,-)$ preserve subspace inclusions. This follows from the universal properties of subspaces, products and the exponential law, \cite[Proposition 2.12]{S09}. Finally, observe that maps of $(\Pi; \Gamma)$-bundles with fibre $F$ define $\Gamma$-maps between associated principal bundles.
 	
 	Conversely, given a principal $(\Pi; \Gamma)$-bundle, $E$, define $\sigma$ to be the map $1 \times_\Pi *: E \times_\Pi F \to E \times_\Pi * := B$. Since $q: E \to E / \Pi$ is an open map and $E$ is a principal $(\Pi; \Gamma)$-bundle, $B$ can be covered by open sets $U$ for which there exists a $\Pi$-homeomorphism $\phi: U \times \Pi \to q^{-1}(U)$ over $U$, noting Remark \ref{echp63}. We define an atlas for $\sigma$ using the set, $\mathcal{D}$, of all such pairs $(U, \phi)$. Indeed, if $(U, \phi) \in \mathcal{D}$, then we can define a chart $(U, \psi)$ by letting $\psi$ equal:
 	
 	\[
 	U \times F \cong (U \times \Pi) \times_\Pi F  \xrightarrow{\phi \times_\Pi 1} q^{-1}(U) \times_\Pi F \cong \sigma^{-1}(U)
 	\]
 	
 	Let $\mathcal{C}$ be the set of all such charts induced by elements of $\mathcal{D}$. There is a homeomorphism $\Pi \times_\Pi F \to F$ given by $(a,x) \to ax$. If $\gamma \in \Gamma$, then under this homeomorphism multiplication by $\gamma$ on $F$ corresponds to:
 	
 	\[
 	\Pi \times_\Pi F \xrightarrow{()^{\gamma^{-1}} \times \gamma} \Pi \times_\Pi F 
 	\]
 	
 	where $a^{\gamma^{-1}} = \gamma a \gamma^{-1}$.

 	  Using these remarks, we want to show that $\mathcal{C}$ defines an atlas for $\sigma$ and condition i) of Definition \ref{echp24} is immediate. Let $(U_i, \phi_i), (U_j, \phi_j) \in \mathcal{D}$ and note that condition ii) of Definition \ref{echp24} reduces to showing that:
 	
 	\[
 F \cong	(\{u\} \times \Pi) \times_\Pi F  \xrightarrow{(\phi_j^{-1} \gamma \phi_i ()^{\gamma}) \times_\Pi 1} (\{gu\} \times \Pi) \times_\Pi F \cong F
 	\]
 	
 	is multiplication by some $\pi \in \Pi$. This follows from the observation that the composite $\phi_j^{-1} \gamma \phi_i ()^{\gamma}: \Pi \to \Pi$ is a $\Pi$-equivariant map and so must be left multiplication by some $\pi \in \Pi$. We define the $(\Pi; \Gamma)$-bundle with fibre $F$ associated to $E$ to be $\sigma$ equipped with the unique maximal atlas, $\mathcal{C}^{'}$, containing $\mathcal{C}$. Finally, the fact that a $\Gamma$-map between principal $(\Pi; \Gamma)$-bundles, $E_1 \to E_2$ induces a map of $(\Pi;\Gamma)$-bundles with fibre $F$, $E_1 \times_\Pi F \to E_2 \times_\Pi F$, follows from the $\Pi$-equivariance of the trivialisations of $E_1$ and $E_2$. 
 	
 	To show that we have defined an equivalence of categories, first note that if $E$ is a principal $(\Pi; \Gamma)$-bundle there is a natural map $E \to Ass(\sigma, \mathcal{C}^{'}) \subset Map(F, E \times_{\Pi} F)$ over $B$, defined by the adjoint of the quotient map $E \times F \to E \times_\Pi F$. The fact that this is a homeomorphism follows by considering its restriction over the open sets $U$ of $B$ corresponding to elements $(U, \psi) \in \mathcal{D}$.
 	
 	Conversely, given a $(\Pi; \Gamma)$-bundle with fibre $F$, $(\sigma: D \to B, \mathcal{C})$, we have a canonical map $Map(F,D) \times_{\Pi} F \to D$ which restricts to $\alpha: Ass(\sigma, \mathcal{C}) \times_{\Pi} F \to D$ over $B$. For the point set topology here, note that subspaces and quotient maps are preserved by pullbacks, \cite[Propositions 2.33, 2.36]{S09}. Again by restricting to charts in $\mathcal{C}$, it can be checked that the map $\alpha$ is a $G$-homeomorphism. In fact, we've seen each chart $(U, \psi) \in \mathcal{C}$ defines a trivialisation of $Ass(\sigma, \mathcal{C})$ which in turn defines a chart of $Ass(\sigma, \mathcal{C}) \times_\Pi F$, which maps under $\alpha$ to $\psi$. So $\alpha$ is an isomorphism of $(\Pi; \Gamma)$-bundles with fibre $F$, as desired.
 \end{proof}

We are particularly interested in the following consequence:

 \begin{lemma} \label{echp62}
 	A $(\Pi; \Gamma)$-bundle with fibre $F$, $(\sigma:D \to B, \mathcal{C})$, satisfies the ECHP iff its associated principal $(\Pi; \Gamma)$-bundle, $p: Ass(\sigma, \mathcal{C}) \to B$, is an $h\Gamma$-fibration.
 \end{lemma}
 
 \begin{proof}
 	We first show that $(\sigma, \mathcal{C})$ satisfies the ECHP iff for every principal $(\Pi; \Gamma)$-bundle, $X$, $p$ has the RLP with respect to the map $X \to X \times I$. 
 
 For this, note that every principal $(\Pi; \Gamma)$-bundle is $\Gamma$-homeomorphic to a principal $(\Pi; \Gamma)$-bundle of the form $Ass(\sigma_1, \mathcal{C}_1)$, by Theorem \ref{echp61}. Therefore, $p$ has the RLP wrt all maps of the form $X \to X \times I$, with $X$ a principal $(\Pi; \Gamma)$-bundle, iff $p$ has the RLP wrt all maps of the form $Ass(\sigma_1, \mathcal{C}_1) \to Ass(\sigma_1, \mathcal{C}_1) \times I$, with $(\sigma_1: D_1  \to B_1, \mathcal{C}_1)$ a $(\Pi; \Gamma)$-bundle with fibre $F$. Consider a lifting problem:
 
  \[
 \begin{tikzcd}
 	Ass(\sigma_1, \mathcal{C}_1) \arrow{r}{\alpha} \arrow[swap]{d}{} & Ass(\sigma, \mathcal{C}) \arrow{d}{p} \\
 	Ass(\sigma_1, \mathcal{C}_1) \times I   \arrow[swap]{r}{L} & B
 \end{tikzcd}
 \]
 
 By Theorem \ref{echp61}, $\alpha$ is equivalent to a map of bundles $f: (\sigma_1, \mathcal{C}_1) \to (\sigma, \mathcal{C})$. Since $\Pi$ acts trivially on $B$ and $B_1 = Ass(\sigma_1, \mathcal{C}_1) / \Pi$, $L$ is equivalent to a homotopy $H: B_1 \times I \to B$ starting at $h:= \alpha / \Pi$. Since the associated principal bundle of $(\sigma_1 \times 1, \mathcal{C} \times 1)$ is $Ass(\sigma_1, \mathcal{C}_1) \times I$, a lift $\tilde{L}: Ass(\sigma_1, \mathcal{C}_1) \times I \to Ass(\sigma, \mathcal{C})$ is equivalent to a map of $(\Pi; \Gamma)$-bundles with fibre $F$, $(\sigma_1 \times 1, \mathcal{C}_1 \times 1) \to (\sigma, \mathcal{C})$ which restricts to $f$ at $0 \in I$. By comparison with Definition \ref{echp20}, we reach the desired conclusion.
 
 Finally, we need to show that $p$ is an $h \Gamma$-fibration iff for every principal $(\Pi; \Gamma)$-bundle, $X$, $p$ has the RLP with respect to the map $X \to X \times I$. One direction is clear. The other follows from the observation that if $Y \to Ass(\sigma, \mathcal{C})$ is an arbitrary map of $\Gamma$-spaces, then $Y$ is also a principal $(\Pi; \Gamma)$-bundle, by Definition \ref{echp21} and Lemma \ref{echp2}. 
 \end{proof}
 
Given Lemma \ref{echp62} it makes sense to define the ECHP for principal $(\Pi; \Gamma)$-bundles as follows:

\begin{definition} \label{echp35}
	We say that a principal $(\Pi;\Gamma)$-bundle, $p: E \to B$, satisfies the ECHP if it is an h$\Gamma$-fibration.
\end{definition}

With these definitions in mind, the main result proved in this paper is the following equivariant covering homotopy property for numerable principal $(\Pi; \Gamma)$-bundles, \cite[Corollary 8]{LM86}, from which everything else follows:

\begin{theorem} \label{echp8}
If $p: E \to B$ is a numerable principal $(\Pi; \Gamma)$-bundle, see Definition \ref{echp31}, then $p$ is an $h\Gamma$-fibration. Here, $\Gamma$ acts on $B$ via the quotient map $\Gamma \to G$.	
\end{theorem}

We deduce that numerable $(\Pi; \Gamma)$-bundles with fibre $F$ are Hurewicz $G$-fibrations in Lemma \ref{echp6}, via the more general result of Lemma \ref{echp5}, and we deduce the ECHP for $(\Pi; \Gamma)$-bundles with fibre $F$ from Theorem \ref{echp8} in Theorem \ref{echp7}.

\begin{remark} \label{echp66}
	The proof of the ECHP in \cite[Corollary 8]{LM86} reduces to (or uses) the case of the less general $G$-$A$ bundles. In particular, the proof uses the fact that the $H$-map $q: G \to G/H$, where $H$ acts by conjugation on $G$, has the RLP with respect to the $H$-maps $X \times \{0\} \to X \times I$, where $X$ is a paracompact $H$-space, \cite[pg. 266]{L82}.  This follows from a paracompact case discussed at the top of \cite[pg.265]{L82}, which had already been proved by Bierstone. See also Lemma \ref{echp6} below. Given this result on the map $q:G \to G/H$, Lashof is able to prove an interesting result on slices homeomorphic to $X \times I$, \cite[Lemma 2.8]{L82}, which he uses to prove the ECHP in the non-paracompact case.
\end{remark}

\section{Proof of the ECHP of Numerable Bundles}

We begin with some routine lemmas that we will use throughout the paper:

	\begin{lemma} \label{echp1}
		Suppose that we have a pullback square of $G$-spaces, where $G$ is a compact Hausdorff group:
		
		\[
		\begin{tikzcd}
			P \arrow{r}{\alpha} \arrow[swap]{d}{\beta} & A \arrow{d}{f} \\
			B \arrow[swap]{r}{\phi} & C \\
		\end{tikzcd}
		\]
		
		Then if for all $G$-orbits $R$ of $A$, $f|_R : R \to C$ is injective, the following square is also a pullback:
		
		\[
		\begin{tikzcd}
			P/G \arrow{r}{\alpha} \arrow[swap]{d}{\beta} & A/G \arrow{d}{f} \\
			B/G \arrow[swap]{r}{\phi} & C/G \\
		\end{tikzcd}
		\]
		
	\end{lemma}
	
	\begin{proof}
		Let $Q$ be the actual pullback and $\sigma: P / G \to Q$ the induced map. Any element of $Q$ is of the form $([a],[b])$, where $[f(a)] = [\phi(b)]$. In particular, there exists a $g$ such that $f(ga) = \phi(b)$ and so $\sigma$ is surjective. Suppose that $(a,b)$ and $(x,y)$ in $P$ are both sent to the same element of $Q$. There exists $g$ such that $gb = y$ and so $f(ga) = f(x)$. Since $f$ is injective on orbits, $(ga,gb) = (x,y)$. So we have a continuous bijection $\sigma: P/G \to Q$. Since we are working in the category of CGWH spaces, $\sigma$ will be a homeomorphism if whenever $K \subset Q$ is compact Hausdorff, $\sigma^{-1}(K)$ is compact. It suffices to show that the preimage of $K$ in $P$ is compact. This is a closed subspace of the preimage under $\pi \times \pi$ of $\tilde{\alpha}(K) \times \tilde{\beta}(K) \subset A / G \times B / G$ in $A \times B$, where $\tilde{\alpha}$ and $\tilde{\beta}$ are the relevant maps from $Q$, so is compact as desired.
	\end{proof}
	
	The next result goes back to at least \cite[Theorem 3.7]{P60}, where it was stated in the context of slices:
	
	\begin{lemma} \label{echp2}
		Let $S$ be an $H$-space. If $f:X \to G \times_H S$ is a $G$-map, then the canonical map $\sigma: G \times_H f^{-1}(S) \to X$ is a homeomorphism. \end{lemma}
	
	\begin{proof}
		Since every orbit of $G \times_H S$ contains an element of $S$, $\sigma$ is surjective. If $g_1 x_1 = g_2 x_2$ with $g_i \in G$ and $x_i \in f^{-1}(S)$, then $g_1 f(x_1) = g_2 f(x_2)$, so $g_1 = g_2 h$ for some $h \in H$ and so $\sigma$ is bijective. If $K \subset X$ is compact Hausdorff then so is $GK \subset X$, and since $f^{-1}(S)$ is closed, $f^{-1}(S) \cap GK$ is compact. So $\sigma^{-1}(K)$ is a closed subspace of a compact space, hence compact. Therefore, $\sigma$ is a homeomorphism.
	\end{proof}
	
	Recall that an open cover $\{U_i\}$ of a $G$-space, $B$, is said to be $G$-numerable if it is locally finite  and, for all $i$, there exists a $G$-equivariant map $\lambda_i: B \to I$ such that $U_i = \lambda^{-1}((0,1])$, where $G$ acts trivially on $I$. We will use the following familiar result, due to Dold, \cite[Theorem 4.8]{D63}, about Hurewicz $G$-fibrations ($:= hG$-fibrations), whose proof is identical to the non-equivariant case, see also \cite[pg.51]{M99}:
	
	\begin{lemma} \label{echp9} Let $p: E \to B$ be a map of $G$-spaces and let $\{U_i\}$ be a $G$-numerable open cover of $B$. Then $p$ is an $hG$-fibration iff $p: p^{-1}(U_i) \to U_i$ is an $hG$-fibration for every $i$.
	\end{lemma}
	
	Finally, we have the following local description of completely regular $G$-spaces:
	
	\begin{lemma} \label{echp33}
		Let $X$ be a completely regular $G$-space. Then, for any $x \in X$ with isotropy group $G_x$, there is a $G_x$-invariant subspace $V_x$ containing $x$, called a slice through $x$, such that the canonical map $G \times_{G_x} V_x \to X$ is a homeomorphism onto an open neighbourhood of the orbit $Gx$.
	\end{lemma}
	
	\begin{proof}
		See \cite[Theorem 5.4]{B72}.
	\end{proof}

We now move on to the proof of the equivariant covering homotopy property. Firstly, recall the definition of a numerable principal $(\Pi; \Gamma)$-bundle:
	
	\begin{definition} \label{echp31}
		A numerable principal $(\Pi; \Gamma)$-bundle is a $\Pi$-free $\Gamma$-space, $E$, such that there exists a $\Gamma$-numerable open cover $\{U_i\}$ of $E$ with $U_i \cong \Gamma \times_{\Lambda_i} S_i$ as $\Gamma$-spaces, where $\Lambda_i \cap \Pi = 1$ and $S_i$ is a $\Lambda_i$-space. Let $B = E / \Pi$ and let $p: E \to B$ be the quotient map.
	\end{definition}

Note that a numerable principal $(\Pi; \Gamma)$-bundle is a principal $(\Pi; \Gamma)$-bundle since if $\Lambda \cap \Pi = 1$, then $\Gamma$ is a completely regular and free $(\Pi \times \Lambda)$-space with action $(\pi, \lambda) \cdot \gamma = \pi \gamma \lambda^{-1}$. 
	
	\begin{remark} \label{echp32}
		If $E$ is a $\Pi$-free $\Gamma$-space which is also completely regular then there is a $\Gamma$-equivariant open cover, $\{U_i\}$, of $E$ with $U_i \cong \Gamma \times_{\Lambda_i} S_i$ by Lemma \ref{echp33}, with $\Lambda_i \cap \Pi = \{e\}$ since $E$ is $\Pi$-free. If $E$ is also paracompact Hausdorff, then there is a $\Gamma$-numerable refinement of this open cover, say $\{V_i\}$, and each $V_i$ is also of the form $\Gamma \times_{\Lambda_i} T_i$ by Lemma \ref{echp2}. Here we are using results on partitions of unity such as \cite[Theorem 4.85]{L10}, \cite[Lemma 1.12]{L82} and the observation that an equivariant partition of unity, as in \cite[Lemma 1.12]{L82}, defines a $\Gamma$-numerable open cover as defined before Lemma \ref{echp9}. So $E$ is a numerable principal $(\Pi; \Gamma)$-bundle.
	\end{remark}

	\begin{lemma} \label{echp3}
		If $\Lambda$ is a subgroup of $\Gamma$ such that $\Lambda \cap \Pi = 1$, then $\Gamma/ \Lambda \to G / \Lambda$ is an $h\Lambda$-fibration (that is a Hurewicz $\Lambda$-fibration with the action of $\Lambda$ on the left).
	\end{lemma}
	
	\begin{proof}
		We will induct on the dimension/no. of connected components of $\Lambda$, with the usual ordering where dimension $n_1$ with $m_1$ components is $\leq$ dimension $n_2$ with $m_2$ components iff $n_1 < n_2$ or $n_1 = n_2$ and $m_1 \leq m_2$. If $\Lambda = 1$, then $\Gamma \to G$ is a non-equivariant $h$-fibration since $\Gamma$ is locally homeomorphic to $\Pi \times U$ over $U \subset G$, and $G$ is paracompact. So suppose that $\Lambda \neq 1$. If we view $\Gamma / \Lambda$ as a $\Pi \Lambda$-space, then around any $x \in \Gamma/\Lambda$ it is locally homeomorphic to $\Pi \Lambda \times_H S$, where $H$ is the isotropy group of $x$ under the $\Pi \Lambda$-action and $S$ is a slice through $x$, by Lemma \ref{echp33}. The isotropy group of $x$ under the $\Gamma$-action is conjugate to $\Lambda$, so $H$ is subconjugate to $\Lambda$ in $\Gamma$. Note also that $H \cap \Pi = 1$. Since $G / \Lambda$ is paracompact Hausdorff, it suffices to prove that $\Pi \Lambda \times_H S \to \Lambda \times_H S$ is an $h\Lambda$-fibration. Indeed, if this is the case then the open cover of $G/ \Lambda$ induced by the $\Lambda \times_H S$ has a $\Lambda$-numerable refinement, on which we can apply Lemma \ref{echp9} (noting that the pullback of an $h\Lambda$-fibration along the inclusion of any subspace of its codomain is also an $h \Lambda$-fibration).   Since $H \cap \Pi = 1$, and $X \times_H Y \cong (X \times Y)/H$ for a suitable $H$-action, Lemma \ref{echp1} implies that we have a pullback square:
		
		\[
		\begin{tikzcd}
			\Pi \Lambda \times_H S \arrow{r} \arrow{d} & \Pi \Lambda \times_H * \arrow{d}{} \\
			\Lambda \times_H S \arrow[swap]{r}{} & \Lambda \times_H * 
		\end{tikzcd}
		\]
		
		Since $H$ is subconjugate to $\Lambda$ in $\Gamma$, either $H$ has a lower dimension/ fewer connected components than $\Lambda$, or $\Lambda \times_H * = *$, and in both cases we can conclude that $\Pi \Lambda \times_H * \to \Lambda \times_H *$ is an $hH$-fibration, by induction in the first case. If we have a lifting problem of $\Pi \Lambda$-spaces:
		
		\[
		\begin{tikzcd}
			X \arrow{r}{f} \arrow{d} & \Pi \Lambda \times_H * \arrow{d}{} \\
			X \times I \arrow[dashed]{ur} \arrow[swap]{r} & \Lambda \times_H * \\
		\end{tikzcd}
		\]
		
		then Lemma \ref{echp2} implies that $X \cong \Pi \Lambda \times_H f^{-1}(*)$, and so the lifting problem can be reduced to a lifting problem of $H$-spaces. Therefore, $\Pi \Lambda \times_H * \to \Lambda \times_H *$ is an $h\Pi \Lambda$-fibration and, therefore, an $h\Lambda$-fibration, as desired.
	\end{proof}
	
	More generally, we can deduce the equivariant covering homotopy property for principal bundles, where note that $\Gamma$ acts on $B$ via the quotient map $\Gamma \to G$:
	
	\begin{theorem} \label{echp4}
		If $p: E \to B$ is a numerable principal $(\Pi; \Gamma)$-bundle, then $p$ is an $h\Gamma$-fibration.
	\end{theorem}
	
	\begin{proof}
		Using Lemma \ref{echp9}, numerability allows us to reduce to the case where $p: \Gamma \times_{\Lambda} S \to G \times_{\Lambda} S$ is a trivial principal $(\Pi; \Gamma)$-bundle. Since $\Lambda \cap \Pi = 1$, Lemma \ref{echp1} implies that we have a pullback:
		
		\[
		\begin{tikzcd}
			\Gamma\times_{\Lambda} S \arrow{r} \arrow{d} & \Gamma \times_{\Lambda} * \arrow{d}{} \\
			G \times_{\Lambda} S \arrow[swap]{r}{} & G \times_{\Lambda} * 
		\end{tikzcd}
		\]
		
		So it suffices to show that $\Gamma / \Lambda \to G / \Lambda$ is an $h\Gamma$-fibration. As above, this follows from reducing a $\Gamma$-equivariant lifting problem:
		
		\[
		\begin{tikzcd}
			X \arrow{r}{f} \arrow{d} & \Gamma \times_{\Lambda} * \arrow{d}{} \\
			X \times I \arrow[dashed]{ur} \arrow[swap]{r} & G \times_{\Lambda} * 
		\end{tikzcd}
		\]
		
		to a $\Lambda$-equivariant lifting problem, using the fact that $X \cong \Gamma \times_{\Lambda} f^{-1}(*)$ by Lemma \ref{echp1}. Then we're done since Lemma \ref{echp3} tells us that $\Gamma / \Lambda \to G / \Lambda$ is an $h\Lambda$-fibration.
	\end{proof}

\begin{remark} \label{echp30}
	Applications of the ECHP appear throughout the literature. For example, \cite[Proposition 4.1]{K25} is the special case where $\Gamma$ is the semi-direct product $K \rtimes_{\alpha} G$, and so we have an extension of compact Lie groups $1 \to K \to \Gamma \to G \to 1$. Note that since $E$ is metrizable it is paracompact, \cite{R69}, and completely regular so defines a numerable principal $(K; \Gamma)$-bundle as in Definition \ref{echp31}, see Remark \ref{echp32}. Therefore, Theorem \ref{echp4} tells us that $E \to E/K$ is an $h\Gamma$-fibration. Via restriction along the splitting $G \to \Gamma$, we see that $E \to E/K$ is also an $hG$-fibration. In \cite{K25}, such results on quotient projections are extended to more general classes of groups than compact Lie groups, such as almost connected metrizable groups, \cite[Theorem 4.2]{K25}.
\end{remark}

In preparation for our discussion of $(\Pi; \Gamma)$-bundles with fibre $F$, we record the following lemma:
	
	\begin{lemma} \label{echp5}
		If $p: E \to B$ is an $h\Gamma$-fibration and $\Pi$ acts trivially on $B$, then $p/ \Pi : E/ \Pi \to B$ is an $hG$-fibration.
	\end{lemma}
	
	\begin{proof}
		Consider a lifting problem of $G$-spaces and $G$-maps:
		
		\[
		\begin{tikzcd}
			X \arrow{r}{} \arrow{d} & E / \Pi \arrow{d}{p / \Pi} \\
			X \times I \arrow[dashed]{ur} \arrow[swap]{r}{g} & B 
		\end{tikzcd}
		\]
		
		Let $P$ denote the pullback of $g$ along $p$, so $P$ is a $\Gamma$-space. Since each $G$-space is also a $\Gamma$-space on which $\Pi$ acts trivially, each $\Pi$-orbit of $X \times I$ consists of a single point and so Lemma \ref{echp1} implies that we have a pullback of $G$-spaces on the right hand side of the following diagram:
		
		\[
		\begin{tikzcd}
			X \arrow{r}{f} \arrow{d} & P / \Pi  \arrow{r} \arrow{d}{q / \Pi} & E / \Pi \arrow{d}{p / \Pi} \\
			X \times I \arrow[dashed]{ur} \arrow[swap]{r}{1} & X \times I \arrow[swap]{r}{g} & B \\
		\end{tikzcd}
		\]
		
		Since $q$, being a pullback of $p$, is an $h \Gamma$-fibration, there is a map $\alpha: P_0 \times I \to P$ over $X \times I$, where $P_0$ is $q^{-1}(X \times \{0\})$ and the restriction of $\alpha$ to $P_0 \times \{0\}$ is the inclusion of $P_0$ into $P$. Note that $f$ factors through the inclusion $P_0 / \Pi \to P / \Pi$, which allows us to define the map $k$ in the following diagram. It suffices to solve the lifting problem:
		
		\[
		\begin{tikzcd}
			X \arrow{r}{k \times \{0\}} \arrow{d} & P_0 / \Pi \times I \arrow{d}{q_0 / \Pi \times 1} \\
			X \times I \arrow[dashed]{ur}{H} \arrow[swap]{r}{1} & X \times I \\
		\end{tikzcd}
		\]
		
		For this define $H(x,t) = (k(x),t)$, so we're done. 
	\end{proof}

	We conclude the paper by relating the results above to $(\Pi; \Gamma)$-bundles with fibre $F$. We begin with:
	
	\begin{definition} \label{echp65}
	A $(\Pi; \Gamma)$-bundle with fibre $F$, $(\sigma, \mathcal{C})$, is called numerable if the associated principal $(\Pi; \Gamma)$-bundle, $Ass(\sigma, \mathcal{C})$, of Theorem \ref{echp61} is numerable, Definition \ref{echp31}. 
	\end{definition}

In particular, if $(\sigma, \mathcal{C})$ is numerable, there is a $G$-numerable open cover of $B$ over which $\sigma$ is locally homeomorphic to bundles of the form $(\Gamma \times_{\Lambda} S) \times_{\Pi} F \to (\Gamma \times_{\Lambda} S) \times_{\Pi} *$. We have:
	
	\begin{lemma} \label{echp6}
		If $(\sigma, \mathcal{C})$ is a numerable $(\Pi; \Gamma)$-bundle with fibre $F$, then $\sigma$ is an $hG$-fibration.
	\end{lemma}

\begin{proof}
	By Theorem \ref{echp61}, $\sigma$ is isomorphic to the bundle $Ass(\sigma, \mathcal{C}) \times_{\Pi} F \to Ass(\sigma, \mathcal{C}) \times_\Pi *$. The composite $Ass(\sigma, \mathcal{C}) \times F \to Ass(\sigma, \mathcal{C}) \times * \to Ass(\sigma, \mathcal{C}) \times_{\Pi} *$ is an $h \Gamma$-fibration, by Theorem \ref{echp4}. Therefore, $\sigma$ is an $hG$-fibration, by Lemma \ref{echp5}.
\end{proof}
	
Finally, by Lemma \ref{echp62} and Theorem \ref{echp4} we have:
	
	\begin{theorem} \label{echp7}
	A numerable $(\Pi; \Gamma)$-bundle with fibre $F$ satisfies the ECHP of Definition \ref{echp20}.
	\end{theorem}

	\bibliography{References}
	\bibliographystyle{alpha}

\end{document}